\documentclass[12pt]{article}
\usepackage{fullpage}

\usepackage[utf8]{inputenc} % allow utf-8 input
\usepackage[T1]{fontenc}    % use 8-bit T1 fonts
\usepackage{hyperref}       % hyperlinks
\usepackage{url}            % simple URL typesetting
\usepackage{booktabs}       % professional-quality tables
\usepackage{amsfonts}       % blackboard math symbols
\usepackage{nicefrac}       % compact symbols for 1/2, etc.
\usepackage{microtype}      % microtypography
\usepackage{lipsum}
\usepackage{fancyhdr}       % header
\usepackage[textsize=scriptsize]{todonotes}

\usepackage{graphicx}       % graphics
\usepackage{subcaption}

\usepackage{hyperref}
\hypersetup{
    colorlinks=true,
    linkcolor=blue,
    filecolor=magenta,      
    urlcolor=cyan,
    citecolor=blue,
}

\usepackage{enumitem}

\usepackage{amsthm}
\usepackage{amsmath}
\newtheorem{theorem}{Theorem}
\newtheorem{corollary}[theorem]{Corollary}
\newtheorem*{problem*}{Open problem}
\newtheorem{claim}[theorem]{Claim}
\newtheorem{observation}[theorem]{Observation}

\newtheorem{lemma}[theorem]{Lemma}

\title{Computing the vertex connectivity of a locally maximal 1-plane graph in linear time
}

\author{Therese Biedl and Karthik Murali \thanks{
  David R. Cheriton School of Computer Science, University of Waterloo, Waterloo, Canada 
  \texttt{\{biedl,k3murali\}@uwaterloo.ca}.  Work of TB supported by  NSERC, FRN RGPIN-2020-03958.
}
}

\begin{document}
\maketitle

\begin{abstract}
It is known that the vertex connectivity of a planar graph can be computed in linear time. We extend this result to the class of locally maximal 1-plane graphs: graphs that have an embedding with at most one crossing per edge such that the endpoints of each pair of crossing edges induce the complete graph $K_4$.
\end{abstract}

\graphicspath{{Figures/}}

\section{Introduction}

One of the oldest problems in graph algorithms is the \emph{connectivity problem}: Given a connected graph $G$, what is the smallest number of vertices (denoted by $\kappa(G)$) that need to be removed from $G$ to make it disconnected? This has applications in security and network reliablity; it is hence important to be able to test efficiently whether $\kappa(G)$ is sufficiently large and to find a corresponding \emph{minimum separating set}, i.e., a smallest set of vertices whose removal disconnects the graph.

Let $G$ be a graph with $n$ vertices and $m$ edges. It is very easy to test in linear time (i.e. $O(m+n)$-time) 
whether $\kappa(G)\geq 1$ by running any graph traversal algorithm, and $\kappa(G)\geq 2$ can easily be tested in linear time by modifying a depth-first search \cite{tarjan1972depth}.  Testing whether $\kappa(G)\geq 3$ in linear time is harder; an algorithm for this was proposed by Hopcroft and Tarjan in 1973 \cite{hopcroft1973dividing}, but it had some errors which were pointed out and corrected in 2000 by Gutwenger and Mutzel \cite{DBLP:conf/gd/GutwengerM00}. For the general question of determining the connectivity of a graph, for the longest time the fastest known algorithms ran in $O(n^2)$ time. For the simplest case when $m \in O(n)$, and hence $\kappa(G) \in O(1)$, Kleitman showed how to determine the connectivity in $O(n^2)$ time  \cite{kleitman1969methods}. For $\kappa(G)=4$ and any $m$, the first $O(n^2)$ algorithm was by Kanevsky and Ramachandran \cite{DBLP:journals/jcss/KanevskyR91}. 
For $\kappa(G) \in O(1)$ and any $m$, the first $O(n^2)$ algorithm was by Nagamochi and Ibaraki \cite{DBLP:journals/algorithmica/NagamochiI92}. 

The last few years have seen some breakthroughs for computing vertex connectivity. Recently, Forster et al.\ \cite{DBLP:conf/soda/ForsterNYSY20} used fast local-cut algorithms to show that when $\kappa(G) \in O(\text{polylogn})$, there is a randomised algorithm that takes time $\tilde{O}(m + nk^3)$\footnote{%
%$\tilde{O}(g(n)) = O(g(n)\log^kn)$ for some constant $k$.} 
$\tilde{O}$ hides poly-logarithmic factors, $\hat{O}$ hides arbitrarily small polynomial factors.}
to compute vertex connectivity, and hence is nearly linear. As for deterministic algorithms, when $k := \kappa(G) < n^{1/8}$, Gao et al.\ \cite{DBLP:journals/corr/abs-1910-07950Gao} used 
%a sub-quadratic time algorithm for computing 
balanced sparse cuts to give a sub-quadratic time algorithm for computing vertex connectivity; the running time is $\widehat{O}(m + \min\{n^{1.75}k^{1+k/2}, n^{1.9}k^{2.5}\})$\addtocounter{footnote}{-1}\footnotemark.
%\footnote{$\widehat{O}(g(n)) = O(g(n)^{1+o(1)}$)}. 

In this paper, we consider the connectivity problem for special kinds of graph classes.  Our work was motivated by a result for \textit{planar graphs}, i.e., graphs that can be embedded on the plane such that no edge crosses another. These graphs have been extensively studied, and their structural properties have been used in the development of many efficient algorithms.
Any simple planar graph has at most $3n-6$ edges, therefore any planar graph $G$ contains a vertex with at most five distinct neighbours, and so $\kappa(G)\leq 5$. The recent results on vertex connectivity show that there exists a near-linear-time randomized algorithm to determine the connectivity of a planar graph. But this result is actually much older:  Eppstein \cite{DBLP:journals/jgaa/Eppstein99} gave in 1999
a deterministic linear-time algorithm to test whether the vertex connectivity of a planar graph is at least $k$ (which by $\kappa(G)\leq 5$ gives a linear-time algorithm to determine the connectivity). This is based on the following approach. Given a planar graph $G$, let the \emph{radialisation} $\Lambda(G)$ be the planar graph obtained by adding a new face vertex inside each face of $G$ and connecting the face vertex to all the vertices on the boundary of the face. Eppstein credits Nishizeki with the observation that any minimum separating set $S$ corresponds to a separating cycle $X$ in $\Lambda(G)$ that separates two vertices of $G$,%
\footnote{Eppstein actually states the condition slightly differently, demanding $X$ to separate two vertices of $G$ in $\Lambda(G)\setminus E(G)$, but we believe this to be incorrect, since the two separated vertices could then be connected via an edge in $G$.} 
has length $2|S|$, visits all vertices of $S$, and uses only edges added during the radialisation.
Determining the connectivity of a planar graph $G$ hence reduces to finding cycles of length $2k$ (for $k=1,2,3,4$) in $\Lambda(G)$ that satisfy these properties; this can be done by modifying Eppstein's subgraph isomorphism testing algorithm slightly \cite{DBLP:journals/jgaa/Eppstein99}.

Our goal in this paper is to generalize these results to a \emph{1-planar graph}, i.e., a graph that can be embedded in the plane with at most one crossing per edge. This graph class was first introduced by Ringel \cite{ringel1965sechsfarbenproblem} and has excited much interest recently, both with respect to the theoretical properties of these graphs and for developing algorithms tailored to this graph class.  See a 2017 overview paper \cite{KLM17} as well some chapters in a recent book \cite{hong2020beyond} for more details.  To our knowledge, no previous results concerning connectivity-testing in 1-planar graphs have appeared.  Any simple 1-planar graph has at most $4n-8$ edges \cite{Bodendiek1983BemerkungenZE}, and therefore trivially $\kappa(G)\leq 7$.  With the recent breakthrough results, there exists a randomized (but complicated) algorithm to determine the connectivity of a 1-planar graph in near-linear time.  But our aim is to instead extend the approach by Eppstein, and therefore gain more insight into the structure of a 1-planar graph, as well as a deterministic linear-time algorithm for the connectivity of a 1-planar graph.

\begin{figure}
  \centering
  \begin{subfigure}[b]{0.4\textwidth}
    \centering
    \includegraphics[width=0.75\textwidth]{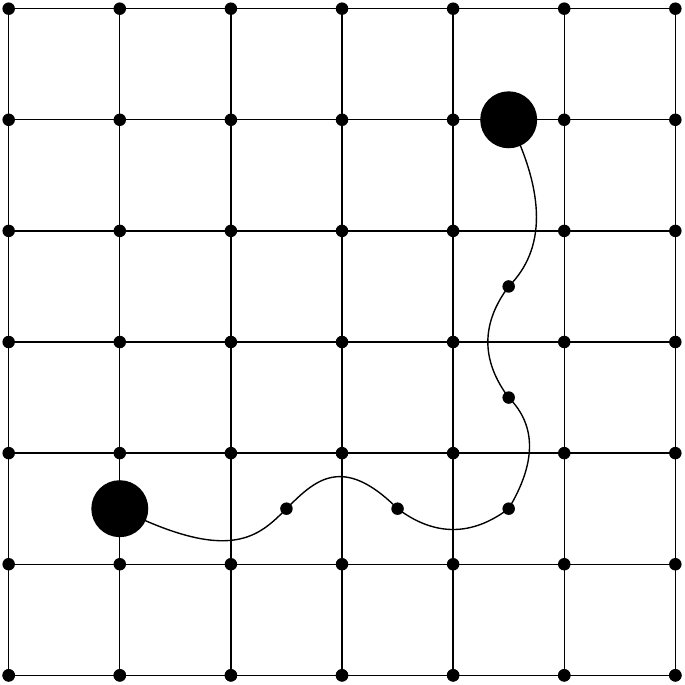}
     \caption{}
    \label{subfig: grid graph}
  \end{subfigure}\hspace{0.1\textwidth}
  \begin{subfigure}[b]{0.4\textwidth}
    \centering
    \includegraphics[width=\textwidth]{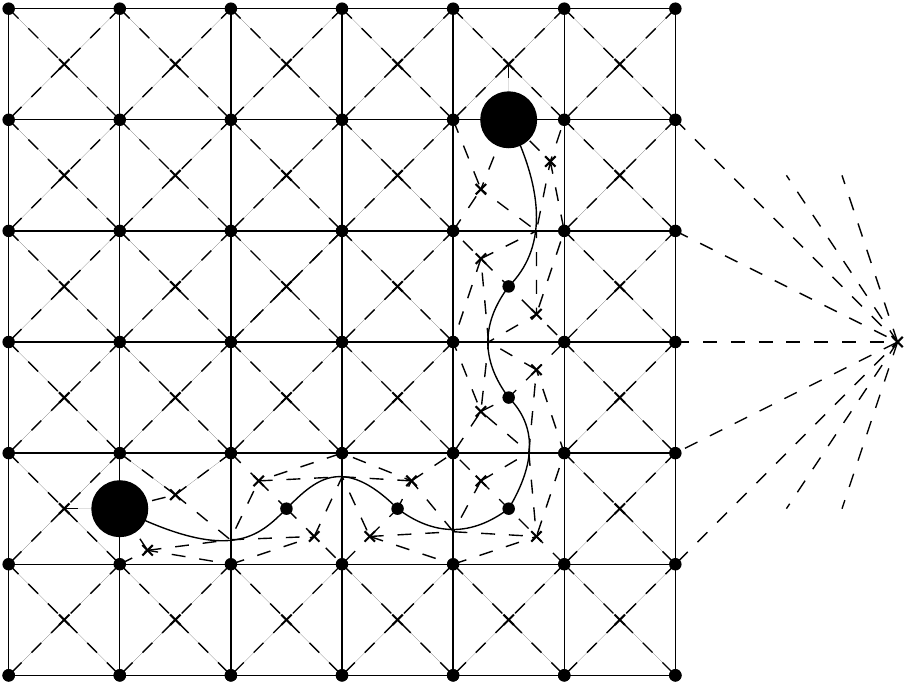}
    \caption{}
    \label{subfig: radial planarisation of the grid graph}
  \end{subfigure}
  \caption{An example to illustrate that a minimum separating of a 1-plane graph $G$ does not always correspond to a bounded length separating cycle in $\Lambda(G)$.}
  \label{fig: grid graph and its planarisation}
\end{figure}

For 1-planar graphs, we define the analogue of radialisation, called \textit{radial planarisation}, where we first planarise the graph by adding dummy vertices at crossing points, and then radialise.
Observe that the straight-forward approach of finding separating cycles in the radial planarisation $\Lambda(G)$ does not always work. Consider Figure \ref{subfig: grid graph} and the two bold vertices, which form a minimum separating set $S$. In $\Lambda(G)$ (Figure \ref{subfig: radial planarisation of the grid graph}), there is no cycle of length $2|S|=4$ that contains these two vertices, and the example can easily be extended to show that in fact no cycle of length $O(1)$ goes through both vertices of $S$. So not every minimum separating set can give rise to a short separating cycle in $\Lambda(G)$. However, as we will show, the approach via radial planarisation \emph{does} work for a subclass of 1-planar graphs. Specifically, call a 1-planar graph \emph{locally maximal} \cite{FHM+20} if it has a 1-planar embedding in the plane where the endpoints of any crossing induce the complete graph $K_4$. As our main result in this paper, we show that for such 1-planar graphs, minimum separating sets indeed correspond to separating cycles (with some other constraints) in $\Lambda(G)$. As such, testing connectivity again reduces to subgraph isomorphism (with some other constraints), and since $\Lambda(G)$ is planar, this can be tested in linear time.

\section{Preliminaries}

(We assume that the reader is familiar with basic concepts of graph theory. See e.g. \cite{diestel2018graph}.) Let $G = (V,E)$ be a graph, which we assume to have no loops since they are irrelevant for connectivity. A \textit{drawing} of $G$ on the plane maps each vertex to a distinct point and each edge to a distinct curve on the plane. We assume that all graph drawings are \textit{good} \cite{Schaefer2013}, i.e., each edge is drawn as a simple non-self-intersecting arc connecting its endpoints, and any two edges intersect at most once, either at a common endpoint or in the interior of the edges. A \textit{planar graph} is a graph that can be drawn on the plane such that no two edges cross; such a drawing is called a \textit{planar drawing}. A planar drawing divides the plane into connected regions called \textit{faces}. A face is identified by its \textit{facial circuit}, which is a directed walk along the boundary of the face such that the face lies to the left of all the edges in the directed walk. This set of facial circuits is called a \textit{planar embedding}. All drawings of a planar graph with the same embedding are \textit{equivalent}, and are said to \textit{respect} the planar embedding. A planar graph with a given planar embedding is called a \textit{plane graph}. 

Let $G$ be a plane graph with a drawing that respects the planar embedding, and let $v$ be a vertex.
The \textit{rotation} at $v$, denoted by $\rho(v)$, is the clockwise sequence of edges incident with $v$. 
If a set of edges $(v,w_1), \dots, (v,w_k)$ incident with $v$ occur in the same order (not necessarily consecutive) in $\rho(v)$, we write $\langle (v,w_1), (v,w_2), \dots, (v,w_k) \rangle \subseteq \rho(v)$.
An \textit{angle} at $v$ is a sequence $\langle u, v, w \rangle$, where $(v,w)$ and $(v,u)$ are consecutive edges in $\rho(v)$. 
A \emph{bigon} is a closed face that contains exactly two angles.
The \textit{rotation system} is the set of rotations at all the vertices of $G$. 

A \textit{crossing} in a drawing of a graph is a pair of edges that intersect at a point that is not a vertex. The point of intersection is called the \textit{crossing point}. An \textit{endpoint of a crossing} is a vertex that is incident with either edge of the crossing. 
Since all graph drawings are good, 
each crossing has four distinct endpoints. 

A \textit{1-planar graph} is a graph that can be drawn on the plane such that each edge is crossed at most once by another edge; such a drawing is called a \textit{1-planar drawing}. A 1-planar graph on $n$ vertices has $O(n)$ edges and crossing points \cite{Bodendiek1983BemerkungenZE, pach1997graphs}. For a given 1-planar drawing of a graph $G$, the \textit{planarisation} of $G$, denoted by $G^\times$, is obtained by placing a \textit{dummy vertex} at each crossing point so that the drawing becomes planar. The set of all facial circuits of a planarised 1-planar drawing is called a \textit{1-planar embedding}. Given a planar graph, one can obtain a planar embedding in linear time (see e.g.~\cite{haeupler2008planarity} and the references therein). In contrast, testing 1-planarity is NP-hard \cite{DBLP:journals/algorithmica/GrigorievB07, korzhik2013minimal}.
Hence for algorithmic purposes, we must assume that a 1-planar graph comes with a given 1-planar embedding; such a graph is called a \textit{1-plane graph}.

\begin{figure}
  \centering
  \begin{subfigure}[b]{0.4\textwidth}
    \centering
    \includegraphics[width=0.85\textwidth]{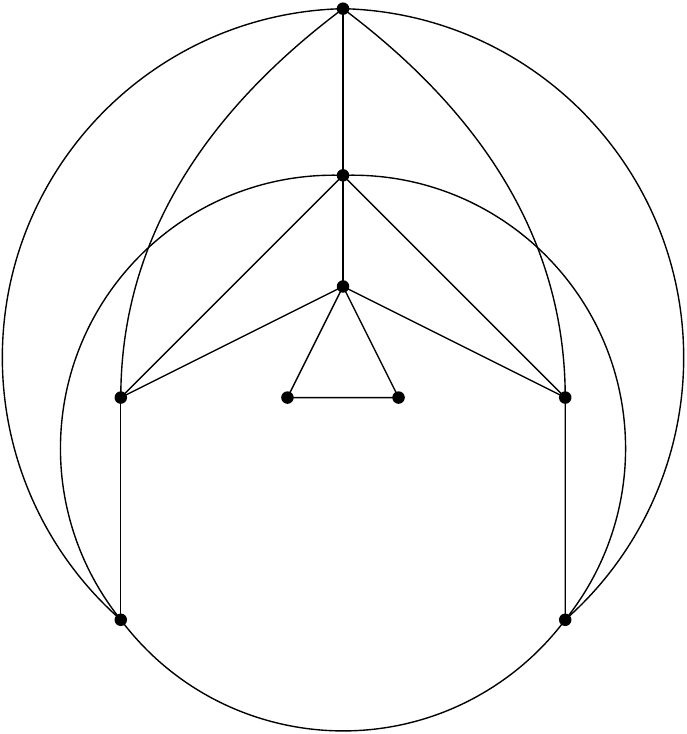}
     \caption{}
    \label{subfig: 1-planar graph}
  \end{subfigure}\hspace{0.1\textwidth}
  \begin{subfigure}[b]{0.4\textwidth}
    \centering
    \includegraphics[width=0.9\textwidth]{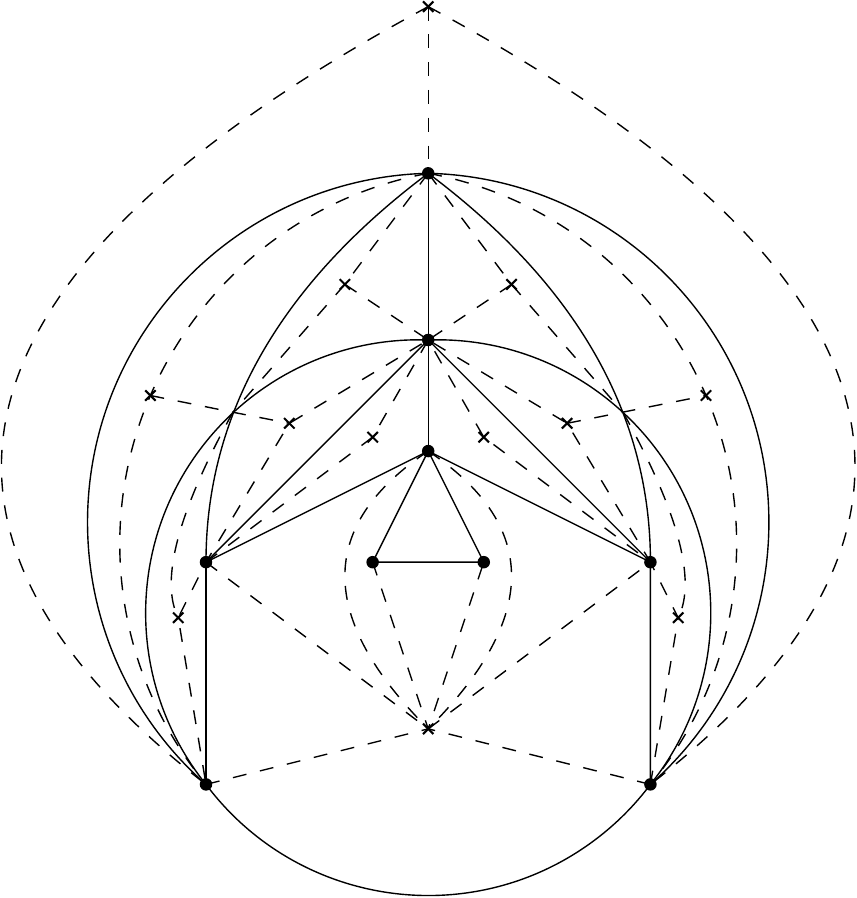}
    \caption{}
    \label{subfig: radial planarisation}
  \end{subfigure}
  \caption{A locally maximal 1-plane graph and its radial planarisation.}
  \label{fig: 1-plane graph and radial planarisation}
\end{figure}

A \emph{locally maximal 1-plane graph} is a 1-plane graph $G$ in which the endpoints of each crossing induce a $K_4$ (Figure \ref{subfig: 1-planar graph}). Thus, if
$\{(u,v), (w,x)\}$ is a crossing in $G$ then all edges $(u,x)$, $(x,v)$, $(v,w)$, $(w,u)$ exist somewhere in $G$. Let $c$ be the crossing point of this crossing.  If edge $e=(u,x)$ is such that $\{u,x,c\}$ bounds a face of $G^\times$, then $e$ is called a \textit{kite edge} of the crossing; the corresponding face is called a \textit{kite face} of the crossing. Note that our definition of locally maximal 1-plane does not require all crossings to have kite edges. However, if edge $e=(u,x)$ does not bound a face incident to $c$, then we can always add an uncrossed edge $e'$ parallel to $e$ (i.e., with the same endpoints $u,x$) to get a kite face $\{u,x,c\}$. Adding parallel edges do not affect the vertex connectivity of a graph, and kite edges can be added in $O(n)$ time since a 1-planar graph has $O(n)$ crossings.  Hence, for the purposes of connectivity testing, we can consider ``locally maximal 1-plane'' to be equivalent to ``has a 1-plane drawing where all crossings have all four kite edges.''

We assume throughout the paper that the input graph $G$ is connected.
A \textit{separating set} of $G$ is a set $S$ of vertices such that $G-S$ is disconnected. Each of the connected components of $G-S$ is called a \textit{flap}. A \textit{separating cycle} is a cycle that is a separating set. A \textit{minimum separating set} is a separating set of minimum cardinality. The size of a minimum separating set is also called the \textit{vertex connectivity} of the graph, and is denoted by $\kappa(G)$. A separating set $S$ is a \textit{minimal separating set} if $S-\{v\}$ is not a separating set for any $v \in S$. The following useful observation on minimal separating sets is easy to verify:

\begin{observation}\label{obs: min sep sets}
Each vertex of a minimal separating set $S$ has a neighbour in each flap of $G-S$.
\end{observation}

\section{The main idea}
Let $G$ be a locally maximal 1-plane graph. The main idea for computing the vertex connectivity of $G$ is to show a correspondence between minimum separating sets of $G$ and shortest \textit{constrained separating cycles} in an auxiliary graph $\Lambda(G)$, called the \textit{radial planarisation} of $G$. We define these terms below.

For any plane graph, the \emph{radial graph} \cite{FT06} is the bipartite graph obtained by placing a \textit{face vertex} in every face and connecting it to every vertex incident with that face.  We use $R(G)$ to denote the radial graph of $G^\times$, and let the \emph{radial planarisation} $\Lambda(G)$ be the union of $G^\times$ and $R(G)$, embedded in such a way that the edges of the radial graph bisect the corresponding angles of $G^\times$ (Figure \ref{fig: 1-plane graph and radial planarisation}).

Not every separating cycle $X$ of $\Lambda(G)$ gives rise to a separating set of $G$ (for example, one flap could only contain dummy vertices), and so we must restrict further the separating cycles that we are searching for.
A cycle $X$ of $\Lambda(G)$ is called a \emph{constrained separating cycle} if it satisfies the following constraints:

\begin{enumerate}[label=($\Psi$\arabic*)]
    \item $\Lambda(G) \setminus X$ has at least two flaps that contain vertices of $G$.
    \item $X$ is a subgraph of $R(G)$.
    \item $X$ does not visit any dummy vertices of $\Lambda(G)$ (i.e., vertices that correspond to crossing points of $G$).
\end{enumerate}

% Note that $\Lambda(G)$ has three types of vertices: vertices of $G$, dummy vertices at the crossing points of $G$, and face vertices. Throughout the paper, the term \textit{vertices of $G$} in $\Lambda(G)$ will refer to only the vertices of $G$, and not dummy vertices or face vertices.

% Observations \ref{obs: Lambda(G) no bigons} and \ref{obs: R(G) bipartite} follow immediately from the construction of $\Lambda(G)$.

% \begin{observation}\label{obs: Lambda(G) no bigons}
% $R(G)$ does not contain bigons.
% \end{observation}

% \begin{observation}\label{obs: R(G) bipartite}
% $R(G)$ is bipartite.
% \end{observation}

% The main idea to find a minimum separating set of $G$ is to find a shortest `constrained separating cycle' in $\Lambda(G)$. A \textit{constrained separating cycle} of $\Lambda(G)$ is a separating cycle $X$ that satisfies the following constraints:

Theorem \ref{thm: the main theorem} captures the correspondence between minimum separating sets of $G$ and shortest constrained separating cycles of $\Lambda(G)$, and is the main result of this paper.

% \begin{theorem}
% The vertices of $G$ on a shortest constrained separating cycle of $\Lambda(G)$ form a minimum separating set of $G$.
% \end{theorem}

\begin{theorem}\label{thm: the main theorem}
Let $G$ be a locally maximal 1-plane graph. Then $G$ has a separating set of size at most $k$ if and only if $\Lambda(G)$ has a constrained separating cycle of length at most $2k$.
\end{theorem}

(In \cite{DBLP:journals/jgaa/Eppstein99}, Eppstein proves a similar theorem for planar graphs, and is the basis for computing vertex connectivity of planar graphs in linear time.) By Theorem \ref{thm: the main theorem}, separating sets of size $k$ in $G$ correspond to constrained separating cycles of length $2k$ in $\Lambda(G)$. (The factor 2 appears because a constrained separating cycle is a subgraph of $R(G)$ by $\Psi 2$.) Therefore, it suffices to compute a shortest constrained separating cycle in $\Lambda(G)$ to compute the vertex connectivity of $G$. It is easy to find such a shortest constrained cycle, based on the planar subgraph isomorphism algorithm developed by Eppstein in \cite{DBLP:journals/jgaa/Eppstein99}. We use $\Lambda(G)$ as the \textit{planar host graph} and a constrained separating cycle as the \textit{pattern graph}. Eppstein's algorithm searches the host graph for an instance of the pattern graph, and outputs the instance if it exists. His algorithm runs in linear time when the size of the pattern graph is bounded. Eppstein also explains briefly in his paper how to modify the algorithm so that it checks for constraints $(\Psi 1)$ and $(\Psi 2)$, and similar easy modifications can be done to ensure $(\Psi 3)$. Since a 1-planar graph is at most 7-connected, by Theorem \ref{thm: the main theorem}, there exists a shortest constrained separating cycle of $\Lambda(G)$ with length at most 14, and therefore of bounded size. Thus, a shortest constrained separating cycle of $\Lambda(G)$ can be computed in linear time. As we will see in Lemma~\ref{lemma: easy direction of the main theorem}, it is very easy to extract a minimum separating set from a constraint separating cycle. By Theorem \ref{thm: the main theorem}, we therefore have:

\begin{theorem}
Let $G$ be a locally maximal 1-plane graph. Then a minimum separating set of $G$ can be computed in linear time.
\end{theorem}

\section{Proof of Theorem \ref{thm: the main theorem}}

The proof of Theorem \ref{thm: the main theorem} proceeds in two directions. 
For convenience, we introduce the following notation. Let $X$ be a separating cycle of $\Lambda(G)$. Then the set of vertices of $G$ on $X$ is denoted by $V_G(X)$. We prove the easy direction (the reverse direction) in the following lemma.

\begin{lemma}\label{lemma: easy direction of the main theorem}
Let $X$ be a constrained separating cycle of $\Lambda(G)$. Then $V_G(X)$ is a separating set of $G$.
\end{lemma}

\begin{proof}
%Let $S$ be the set of vertices of $G$ on $X$. 
From $(\Psi 1)$, there exist two vertices $u, w \in V(G)$ that belong to different flaps of $\Lambda(G) \setminus X$. Consider a simple path $P$ connecting $u$ and $w$ in $G$. Let $P'$ be the path in $\Lambda(G)$ corresponding to $P$ defined as follows: for each edge $e \in P$, (a) if $e$ is uncrossed, add $e$ to $P'$; (b) if $e$ is crossed, add the two edges of $\Lambda(G)$ corresponding to $e$ to $P'$. Note that $P'$ may contain both original vertices of $G$ and dummy vertices corresponding to crossing points. Since $X$ separates $u$ and $w$ in $\Lambda(G)$, the path $P'$ intersects $X$ at some vertex $v \in X$, where $v \notin \{u,w\}$. From $(\Psi 3)$, $v$ is not a dummy vertex of $\Lambda(G)$. Therefore, $v$ is a vertex of $G$, and hence $v \in V_G(X)$. Since $P$ is arbitrary, every path of $G$ connecting $u$ and $w$ contains a vertex of $V_G(X)$ in its interior. This shows that $V_G(X)$ is a separating set of $G$.
\end{proof}

% \begin{corollary}\label{cor: vertices of G on X form a separating set}
% If $\Lambda(G)$ has a constrained separating cycle $X$ of length $2k$, then $G$ has a separating set of size $k$.
% \end{corollary}

% \begin{proof}
% Since $R(G)$ is bipartite (Observation \ref{obs: R(G) bipartite}) and $X$ is a subgraph of $R(G)$ ($\Psi 2$), the vertices of $X$ alternate between vertices of $G$ and face vertices ($\Psi 2)$. The rest of the proof follows from Claim \ref{claim: easy direction of the main theorem}.
% \end{proof}

% \section{From separating sets to separating cycles}

We prove the forward direction of Theorem \ref{thm: the main theorem} in Lemma \ref{lemma: hard direction of the main theorem}. 

\begin{lemma}\label{lemma: hard direction of the main theorem}
Let $S$ be a separating set of $G$. Then $\Lambda(G)$ has a constrained separating cycle $X$ such that $V_G(X) \subseteq S$.
\end{lemma}

\begin{proof}
We can assume that $S$ is minimal, since otherwise we can always choose a subset of $S$ that is a minimal separating set. The plan is to take a maximal path that alternates between face vertices and vertices of $S$, and then to show that extending it would give a cycle that has a vertex of $G$ inside and a vertex of $G$ outside. Note that such a cycle satisfies all the three constraints $(\Psi 1)$, $(\Psi 2)$ and $(\Psi 3)$ with the property that $V_G(X) \subseteq S$.

To obtain such a cycle, we first mark face vertices of $\Lambda(G)$ that have some suitable properties. Let $\langle u,v,w \rangle$ be an angle in $G^\times$ such that $v \in S$. (Note that either $u$ or $w$ could be a dummy vertex of $\Lambda(G)$.) Let $f$ be the face vertex at the angle. If $u$ and $w$ are both vertices of $G$, and they belong to different flaps of $G-S$, then we mark $f$.  We also mark $f$ if one of $u,w$ is a vertex of $G$ and in $S$. Put differently, face vertex $f$ is marked if it either is incident to an edge connecting two vertices of $S$, or it witnesses a transition from one flap to another.

\begin{claim}\label{claim: marked face vertex degree 2}
Each marked face vertex 
is incident with at least two edges whose other endpoint is in $S$.
\end{claim}

\begin{proof}
If a face vertex $f$ is marked, then there is a vertex $v\in S$ and an angle $\langle u,v,w \rangle$ in $G^\times$ that satisfies the conditions for marking. If either $u \in S$ or $w \in S$, we are done. Otherwise, $u$ and $w$ are vertices of $G$ that are in different flaps of $G-S$; in particular they are not dummy vertices. Therefore
%, $\langle u,v,w \rangle$ is also an angle in $G$, and 
the face $F$ of $G$ corresponding to $f$ is not a kite face. 
Consider a directed walk along the boundary of $F$, starting from the edge $(v,w)$ in the direction $v$ to $w$, and ending at edge $(u,v)$. The walk does not see any dummy vertices since $F$ is not a kite face and the graph is locally maximal 1-planar. As $u$ and $w$ belong to different flaps, the walk (which begins in the flap of $w$ and ends in the flap of $u$) must visit some vertex of $S$ inbetween.  Thus the walk either sees a vertex of $S$ different from $v$, or sees $v$ again at another angle of the face. In either case, there must be two edges connecting $f$ to vertices of $S$.
\end{proof}

In Claim \ref{claim: locally maximal case analysis}, we show that each vertex $v$ of $S$ also has two incident edges connecting it to marked face vertices.   But in fact, we need a stronger claim, namely, that these two edges can be restricted within their location of $\rho(v)$ to lie 
on opposite sides of two edges to two flaps. To clarify this, we introduce some notation. Let $\rho_G(v)$ and $\rho_\Lambda(v)$ be the rotation at a vertex $v$ in $G$ and $\Lambda(G)$ respectively. Let $(u,v)$ be an edge of $G$. The notation $(u,\bar{v})$ refers to the edge in $\Lambda(G)$ that is incident with $u$ and corresponds to $(u,v)$ in $G$. That is, $\bar{v} = v$ if $(u,v)$ is uncrossed; else $\bar{v}$ is the dummy vertex on $(u,v)$.

\begin{claim}\label{claim: locally maximal case analysis}
Let $\phi_1$ and $\phi_2$ be two flaps of $G-S$. Let $(v,t_1)$ and $(v,t_2)$ be two edges such that $v \in S$, $t_1 \in \phi_1$ and $t_2 \in \phi_2$. Then there exist a marked face vertex $f$ such that $\langle (v,\bar{t_1}), (v,f), (v,\bar{t_2}) \rangle \subseteq \rho_\Lambda(v)$.
\end{claim}
 
\begin{proof}
Take the closest pair of edges $(v,w_1)$, $(v,w_2)$ that lie between $(v,t_1)$, $(v,t_2)$ in $\rho_G(v)$ such that $w_1 \in \phi_1$ and $w_2 \in \phi_2$. More precisely, take a pair of edges $(v,w_1),(v,w_2)$ such that $\langle (v,t_1)$, $(v,w_1)$, $(v,w_2)$, $(v,t_2) \rangle \subseteq \rho_G(v)$ (where possibly $t_1=w_1$ and/or $w_2=t_2$) and such that there is no vertex $v' \in \phi_1 \cup \phi_2$ with $\langle (v,w_1)$, $(v,v')$, $(v,w_2) \rangle \subseteq \rho_G(v)$. It suffices to show that there is a marked face vertex $f$ such that $\langle (v,\bar{w_1})$, $(v,f)$, $(v, \bar{w_2}) \rangle \subseteq \rho_\Lambda(v)$.
We have two cases:

\medskip\noindent\textbf{Case 1:
$(v,w_1)$ and $(v,w_2)$ are consecutive in $\rho_G(v)$.} We claim that both $(v,w_1)$ and $(v,w_2)$ are uncrossed. For contradiction, suppose that $(v,w_1)$ is crossed by an edge $(u,u')$. Then there exist kite edges $(v,u)$, $(v,u')$ before and after $(v,w_1)$ in $\rho_G(v)$. Hence (up to renaming) $u' = w_2$. Since $(w_1, u')$ is also a kite edge, $w_1$ and $w_2$ are adjacent, which contradicts that they are in different flaps. Therefore, $(v,w_1)$ is uncrossed, and similarly $(v,w_2)$ is also uncrossed. This implies that $\langle w_1, v, w_2 \rangle$ is an angle at $v$ in $G$. Since $w_1$ and $w_2$ belong to different flaps, the face vertex $f$ at this angle is marked as desired.

\medskip\noindent\textbf{Case 2: 
$(v,w_1)$ and $(v,w_2)$ are not consecutive in $\rho_G(v)$.} Then there is a vertex $v'$ such that $\langle (v,w_1), (v,v'), (v,w_2) \rangle \subseteq \rho_G(v)$. Choose $(v,v')$ to be the first edge that comes after $(v,w_1)$ in $\rho_G(v)$.
By the choice of $w_1$ and $w_2$, $v' \notin \phi_1 \cup \phi_2$. We have two subcases.
If $(v,v')$ is uncrossed, then $\langle w_1, v, v' \rangle$ is an angle at $v$ and $v$ is not a dummy-vertex. Since $v' \notin \phi_1$, it either belongs to $S$ or to some other flap; either way the face vertex $f$ at this angle is marked and we are done.  Now suppose that $(v,v')$ is crossed by an edge $(u,u')$ at $c$. Then there exist kite edges $(v,u)$ and $(v,u')$ before and after $(v,v')$ in $\rho_G(v)$. By the choice of $v'$, we have (up to renaming) $u = w_1$. So $u'$ is a neighbour of $w_1$, which means that it either belongs to the same flap $\phi_1$ or to $S$. In particular $u'\neq w_2$.   But this implies $u'\not\in \phi_1$ by choice of $w_1,w_2$, so $u'\in S$.
Therefore $\langle c,v,u' \rangle$ is an angle in $G^\times$ with $u', v \in S$ and the face vertex $f$ at this angle is marked as desired.
\end{proof}

\paragraph{Constructing the constrained separating cycle.}

We now show how to construct the desired constrained separating cycle in $\Lambda(G)$ using only vertices of $S$ and marked face vertices. Let $v_1$ be an arbitrary vertex of $S$. Let $P = v_1\dots v_k$ be a simple path that alternates between marked face vertices and vertices of $S$ and that is maximal in the following sense: $v_k \in S$, and 
for any marked face vertex $v_{k+1}$ and any vertex $v_{k+2}$ in $S$ the extension
$P \cup \{(v_k,v_{k+1}),(v_{k+1},v_{k+2})\}$ would not be a simple path. 
%there does not exist a marked face vertex, say $v_{k+1}$, and a vertex of $S$, say $v_{k+2}$, such that $P \cup \{(v_k,v_{k+1}),(v_{k+1},v_{k+2})\}$ is a simple path. 
% TB: It's always a bit dangerous to give names to something that doesn't exist, so reworded this to make it a positive statement.
Since $v_k \in S$, there exist vertices $t_1 \in \phi_1$ and $t_2 \in \phi_2$ that are adjacent to $v_k$ (Observation \ref{obs: min sep sets}). By symmetry, we may assume that $\langle (v_k,\bar{t_2}), (v_k, v_{k-1}), (v_k,\bar{t_1}) \rangle \subseteq \rho_\Lambda(G)$. From Claim \ref{claim: locally maximal case analysis}, there is a marked face vertex $f$ such that $\langle (v_k,\bar{t_1}), (v_k,f), (v_k,\bar{t_2}) \rangle \subseteq \rho_\Lambda(G)$. 

\begin{figure}
  \centering
  \begin{subfigure}[b]{0.3\textwidth}
    \centering
    \includegraphics[scale = 0.85]{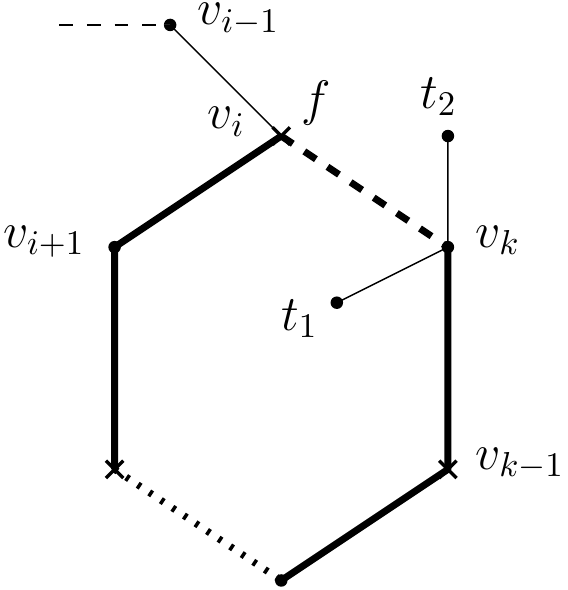}
     \caption{}
    \label{subfig: Separating cycles case 1}
  \end{subfigure}
  \begin{subfigure}[b]{0.3\textwidth}
    \centering
    \includegraphics[scale = 0.85]{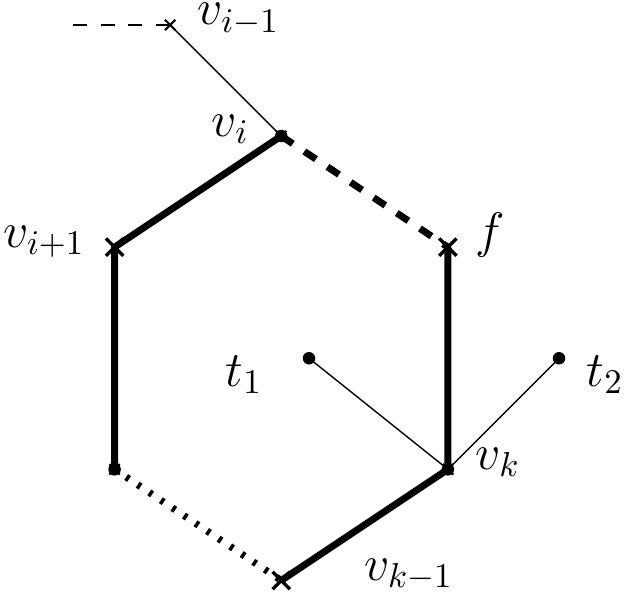}
    \caption{}
    \label{subfig: Separating cycles case 2}
  \end{subfigure}
  \begin{subfigure}[b]{0.3\textwidth}
    \centering
    \includegraphics[scale = 0.85]{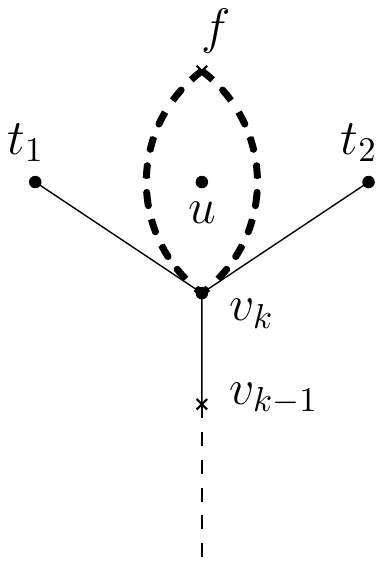}
    \caption{}
    \label{subfig: Separating cycles case 3}
  \end{subfigure}
  \caption{Constructing a constrained separating cycle in $\Lambda(G)$}
  \label{fig: Constructing a marked separating cycle}
\end{figure}

\begin{itemize}
    \item Suppose that $f = v_i\in P$, for some $1 \leq i \leq k-1$ (Figure \ref{subfig: Separating cycles case 1}). Then $X:= v_iv_{i+1}\dots v_kv_i$ is a cycle that separates $t_1$ and $t_2$ (since they are on either side of the cycle).

    \item Suppose that $f \notin P$. 
    By Claim \ref{claim: marked face vertex degree 2}, $f$ has two edges $e,e'$ to vertices of $S$. We may assume that $e$ is the edge $(v_k,f)$ that brought us to $f$, and consider extending $P$ via $e$ and $e'$.  Since this gives a non-simple path by choice of $P$, we have $e'=(f,v_i)$ 
%    Since $P$ is maximal,  and each face vertex has two edges to vertices of $S$  $f$ has an edge to a vertex $v_i \in P$, 
% TB: That one was way too dense, and also not clear why you aren't re-using the edge to $v_k$ that you just came from.
for some $1 \leq i \leq k$. If $1 \leq i \leq k-1$ (Figure \ref{subfig: Separating cycles case 2}), then $X := v_iv_{i+1}\dots v_kfv_i$ is a cycle that separates $t_1$ and $t_2$. If $i = k$, then $e$ and $e'$ form a 2-cycle, and there must exist a vertex $u \in G$ inside the 2-cycle 
because the radial graph $R(G)$ does not contain bigons  (Figure \ref{subfig: Separating cycles case 3}). 
Therefore, $X:= v_kfv_k$ is a cycle separating $u$ from $t_1$ (and $t_2$).
\end{itemize}

Note that by construction, $X$ separates two vertices of $G$, is a subgraph of $R(G)$, does not visit dummy vertices, and all vertices of $G$ on $X$ belong to $S$. This proves the lemma.
\end{proof}

\iffalse
\begin{corollary}
The vertex connectivity of a locally maximal 1-plane graph is half the length of a shortest constrained separating cycle of $\Lambda(G)$.
\end{corollary}

\section{Computing a shortest constrained separating cycle}

As $\kappa(G) \leq 7$ for a 1-planar graph, the length of a shortest constrained separating cycle in $\Lambda(G)$ is at most 14. We can find such a cycle in $O(n)$ time using Eppstein's algorithm for separating cycles \cite[Theorem 6]{DBLP:journals/jgaa/Eppstein99} after a small modification to accommodate the constraints. Eppstein's algorithm only searches $\Lambda(G)$ for cycles that satisfy $(\Psi 1)$ and $(\Psi 2)$. To enforce $(\Psi 3)$, we delete all dummy vertices from $\Lambda(G)$ as a pre-processing step of the algorithm. This step takes $O(n)$ time since a 1-planar graph has only $O(n)$ crossing points.
\fi 

\section{Outlook}

In this paper, we studied how to compute the vertex connectivity of a locally maximal 1-plane graph $G$.  We showed that any minimum separating set $S$ corresponds to a separating cycle of length $2|S|$ (and with some other properties) in $\Lambda(G)$. By appealing to known subgraph isomorphism results, we can find such a cycle, and with it a minimum separating set, in linear time.

\begin{figure}
  \centering
  \begin{subfigure}[b]{0.4\textwidth}
    \centering
    \includegraphics[width=\textwidth]{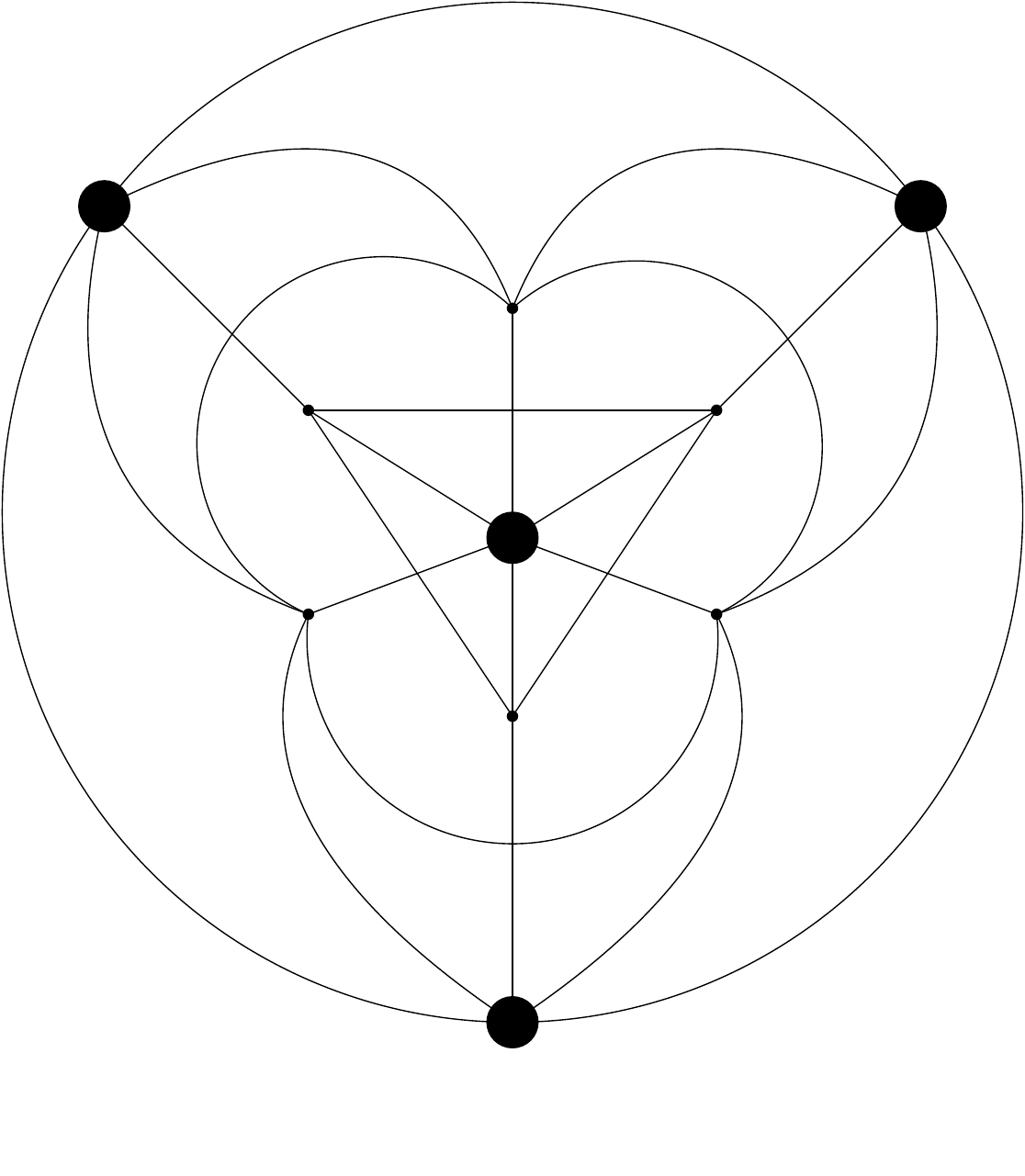}
     \caption{}
    \label{subfig: arrow crossings}
  \end{subfigure}\hspace{0.1\textwidth}
  \begin{subfigure}[b]{0.4\textwidth}
    \centering
    \includegraphics[width=\textwidth]{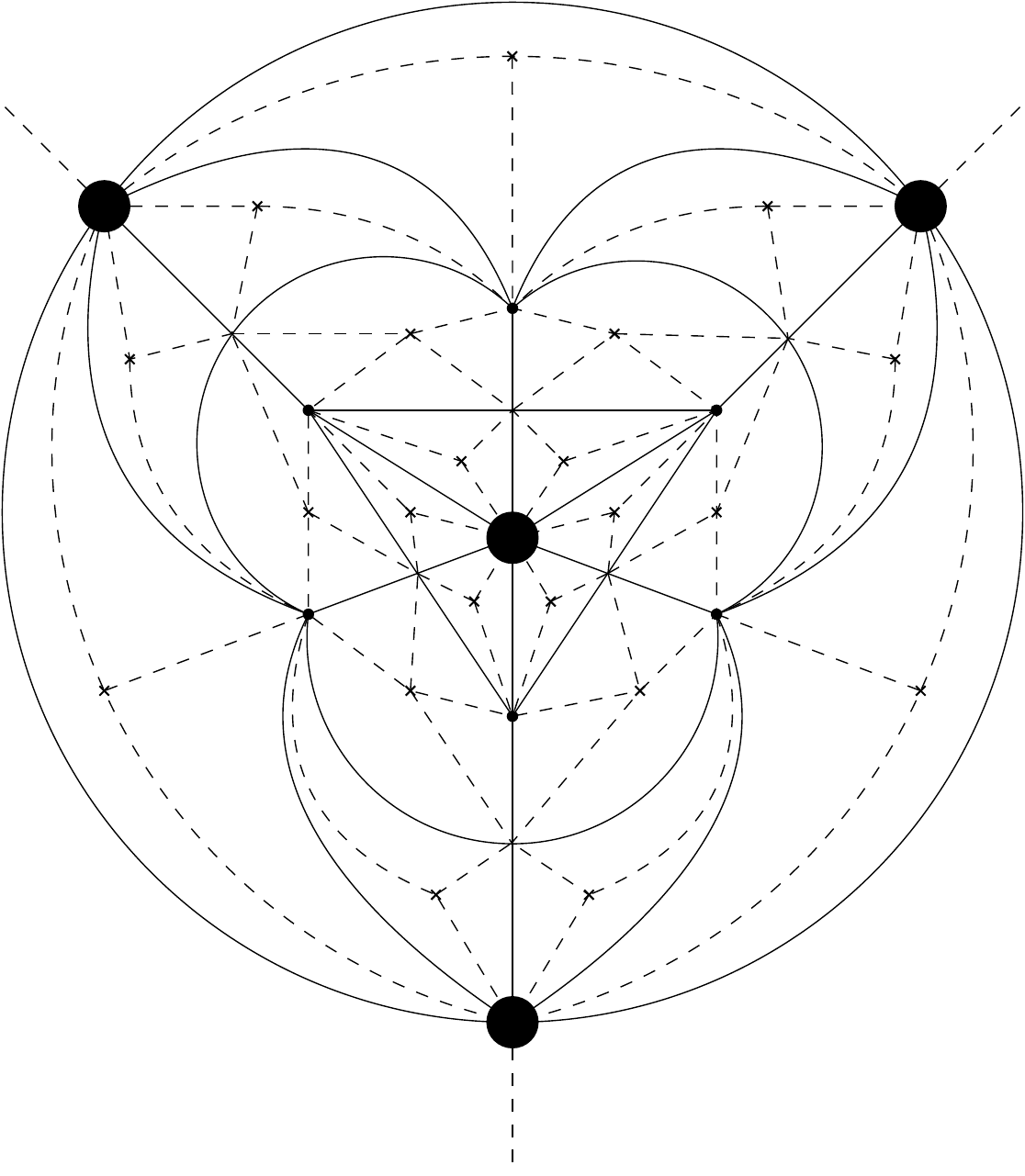}
    \caption{}
    \label{subfig: planarised arrow crossing}
  \end{subfigure}
  \caption{An example of a 4-connected 1-plane graph that does not have a separating 8-cycle in its radial planarisation passing through a minimum separating set (bold vertices).}
  \label{fig: arrow crossings and its planarisation}
\end{figure}

The most natural open question is to extend the result to more subclasses of 1-planar graphs.  We required that the 1-plane graph was locally maximal, i.e., for every crossing all four kite edges exist.  We believe that the techniques used here can be extended (after a slight modification to $\Lambda(G)$) even to the case when there are only three kite edges, and even to the case when there are only two kite edges and they do not have a common endpoint.  (Details will be given in a forthcoming publication.)  The case when there are only two kite edges with a common endpoint appears problematic, since the example in Figure \ref{fig: arrow crossings and its planarisation} shows that (at least for our way of defining the radial planarisation) there exists a minimum separating set of size 4 (bold vertices), but no 8-cycle in the radial planarisation connecting them.  Even more problematic is the situation when there are no kite edges at all, as illustrated in Figure \ref{fig: grid graph and its planarisation}.  We suspect that radically different techniques, not based on radial planarisation, will be needed to find the minimum separating set of all 1-plane graphs in $O(n)$ time.

Finally, looking more broadly, can we find linear-time connectivity testing algorithms for other graph classes that are defined via some geometric properties, such as $k$-planar graphs, fan-planar graphs, or sphere-of-influence graphs?

\bibliographystyle{alpha}  
\bibliography{references}

\end{document}